\documentclass[12pt]{amsart}
\usepackage{graphics}
\usepackage{amsfonts,amssymb,color}
\usepackage[mathscr]{eucal}
\usepackage{amsmath, amsthm}
\usepackage{mathrsfs}
\usepackage{amsbsy}
\usepackage{dsfont}
\usepackage{bbm}
\usepackage{wasysym}
\usepackage{stmaryrd}
\usepackage{url}

\input xypic
\xyoption{all}

\makeindex
\makeglossary

\begin{document}
\baselineskip = 16pt
\newcommand \C{{\mathbb C}}
\newcommand \ZZ {{\mathbb Z}}
\newcommand \NN {{\mathbb N}}
\newcommand \QQ{{\mathbb Q}}
\newcommand \RR {{\mathbb R}}
\newcommand \PR {{\mathbb P}}
\newcommand \AF {{\mathbb A}}
\newcommand \GG {{\mathbb G}}
\newcommand \bcA {{\mathscr A}}
\newcommand \bcC {{\mathscr C}}
\newcommand \bcD {{\mathscr D}}
\newcommand \bcF {{\mathscr F}}
\newcommand \bcG {{\mathscr G}}
\newcommand \bcH {{\mathscr H}}
\newcommand \bcM {{\mathscr M}}
\newcommand \bcJ {{\mathscr J}}
\newcommand \bcL {{\mathscr L}}
\newcommand \bcO {{\mathscr O}}
\newcommand \bcP {{\mathscr P}}
\newcommand \bcQ {{\mathscr Q}}
\newcommand \bcR {{\mathscr R}}
\newcommand \bcS {{\mathscr S}}
\newcommand \bcV {{\mathscr V}}
\newcommand \bcW {{\mathscr W}}
\newcommand \bcX {{\mathscr X}}
\newcommand \bcY {{\mathscr Y}}
\newcommand \bcZ {{\mathscr Z}}
\newcommand \goa {{\mathfrak a}}
\newcommand \gob {{\mathfrak b}}
\newcommand \goc {{\mathfrak c}}
\newcommand \gom {{\mathfrak m}}
\newcommand \gon {{\mathfrak n}}
\newcommand \gop {{\mathfrak p}}
\newcommand \goq {{\mathfrak q}}
\newcommand \goQ {{\mathfrak Q}}
\newcommand \goP {{\mathfrak P}}
\newcommand \goM {{\mathfrak M}}
\newcommand \goN {{\mathfrak N}}
\newcommand \uno {{\mathbbm 1}}
\newcommand \Le {{\mathbbm L}}
\newcommand \Spec {{\rm {Spec}}}
\newcommand \Gr {{\rm {Gr}}}
\newcommand \Pic {{\rm {Pic}}}
\newcommand \Jac {{{J}}}
\newcommand \Alb {{\rm {Alb}}}
\newcommand \Corr {{Corr}}
\newcommand \Chow {{\mathscr C}}
\newcommand \Sym {{\rm {Sym}}}
\newcommand \Prym {{\rm {Prym}}}
\newcommand \cha {{\rm {char}}}
\newcommand \eff {{\rm {eff}}}
\newcommand \tr {{\rm {tr}}}
\newcommand \Tr {{\rm {Tr}}}
\newcommand \pr {{\rm {pr}}}
\newcommand \ev {{\it {ev}}}
\newcommand \cl {{\rm {cl}}}
\newcommand \interior {{\rm {Int}}}
\newcommand \sep {{\rm {sep}}}
\newcommand \td {{\rm {tdeg}}}
\newcommand \alg {{\rm {alg}}}
\newcommand \im {{\rm im}}
\newcommand \gr {{\rm {gr}}}
\newcommand \op {{\rm op}}
\newcommand \Hom {{\rm Hom}}
\newcommand \Hilb {{\rm Hilb}}
\newcommand \Sch {{\mathscr S\! }{\it ch}}
\newcommand \cHilb {{\mathscr H\! }{\it ilb}}
\newcommand \cHom {{\mathscr H\! }{\it om}}
\newcommand \colim {{{\rm colim}\, }} 
\newcommand \End {{\rm {End}}}
\newcommand \coker {{\rm {coker}}}
\newcommand \id {{\rm {id}}}
\newcommand \van {{\rm {van}}}
\newcommand \spc {{\rm {sp}}}
\newcommand \Ob {{\rm Ob}}
\newcommand \Aut {{\rm Aut}}
\newcommand \cor {{\rm {cor}}}
\newcommand \Cor {{\it {Corr}}}
\newcommand \res {{\rm {res}}}
\newcommand \red {{\rm{red}}}
\newcommand \Gal {{\rm {Gal}}}
\newcommand \PGL {{\rm {PGL}}}
\newcommand \Bl {{\rm {Bl}}}
\newcommand \Sing {{\rm {Sing}}}
\newcommand \spn {{\rm {span}}}
\newcommand \Nm {{\rm {Nm}}}
\newcommand \inv {{\rm {inv}}}
\newcommand \codim {{\rm {codim}}}
\newcommand \Div{{\rm{Div}}}
\newcommand \sg {{\Sigma }}
\newcommand \DM {{\sf DM}}
\newcommand \Gm {{{\mathbb G}_{\rm m}}}
\newcommand \tame {\rm {tame }}
\newcommand \znak {{\natural }}
\newcommand \lra {\longrightarrow}
\newcommand \hra {\hookrightarrow}
\newcommand \rra {\rightrightarrows}
\newcommand \ord {{\rm {ord}}}
\newcommand \Rat {{\mathscr Rat}}
\newcommand \rd {{\rm {red}}}
\newcommand \bSpec {{\bf {Spec}}}
\newcommand \Proj {{\rm {Proj}}}
\newcommand \pdiv {{\rm {div}}}
\newcommand \CH {{\rm {CH}}}
\newcommand \wt {\widetilde }
\newcommand \ac {\acute }
\newcommand \ch {\check }
\newcommand \ol {\overline }
\newcommand \Th {\Theta}
\newcommand \cAb {{\mathscr A\! }{\it b}}


\newenvironment{pf}{\par\noindent{\em Proof}.}{\hfill\framebox(6,6)
\par\medskip}

\newtheorem{theorem}[subsection]{Theorem}
\newtheorem{conjecture}[subsection]{Conjecture}
\newtheorem{proposition}[subsection]{Proposition}
\newtheorem{lemma}[subsection]{Lemma}
\newtheorem{remark}[subsection]{Remark}
\newtheorem{remarks}[subsection]{Remarks}
\newtheorem{definition}[subsection]{Definition}
\newtheorem{corollary}[subsection]{Corollary}
\newtheorem{example}[subsection]{Example}
\newtheorem{examples}[subsection]{examples}

\title{Push-forwards of Chow groups of smooth ample divisors, with an emphasis  on Jacobian varieties.}
\author{ Kalyan Banerjee, Jaya NN  Iyer}
\address{School of Mathematics, Tata Institute of Fundamental research, 	
Homi Bhabha Road, Colaba, Mumbai 400001}
\email{kalyan@math.tifr.res.in}

\address{The Institute of Mathematical Sciences, CIT
Campus, Taramani, Chennai 600113, India}
\email{jniyer@imsc.res.in}

\author{James D. Lewis}

\address{Department of Mathematics, University of Alberta, Edmonton, AB, Canada T6G 2G1}
\email{lewisjd@ualberta.ca}

\footnotetext{Mathematics Classification Number: 14C15,14C20,15C25, 14C35,14F42}

\begin{abstract}
With a homological Lefschetz conjecture in mind,  we prove the injectivity of the push-forward morphism on rational Chow groups, 
induced by the closed embedding of an ample divisor linearly equivalent to a higher multiple of the 
Theta divisor inside the Jacobian variety $J(C)$, where $C$ is a smooth irreducible complex projective curve.
\end{abstract}

\maketitle

\setcounter{tocdepth}{1}
\tableofcontents

\section{Introduction}

Suppose $X$ is a smooth projective variety defined over the field of complex numbers. Let $D\subset X$ be an ample smooth divisor on $X$.
Denote the closed embedding, $j:D\hookrightarrow X$. Consider the push-forward homomorphism  on Chow groups induced by $j$:
$$
j_*:\CH_k(D;\QQ) \rightarrow \CH_k(X; \QQ),
$$
for $k\geq 0$.
In this paper, we investigate the kernel of the morphism $j_*$.
This question is motivated by the following results and conjectures.
When Chow groups are replaced by the singular homology of a smooth projective variety over $\C$, 
the (dual of the) Lefschetz hyperplane theorem gives an isomorphism of the pushforward map:
$$
j_*:H_k(D,\ZZ) \to H_k(X,\ZZ)
$$
for $k \,<\, \dim D$, and surjectivity when $k= \dim D$.
M. Nori \cite[Conjecture 7.2.5]{Nori} conjectured the following:
\begin{conjecture}\label{Nor}
Suppose $D$ is a very general smooth ample divisor on $X$, of sufficiently large degree. Then the restriction map (the refined Gysin map, \cite{Fulton}):
$$
j^*:\CH^p(X;\QQ) \rightarrow \CH^p(D;\QQ)
$$
is an isomorphism, for $p<\dim D$ and is injective, for $p=\dim D$. 
\end{conjecture}

More generally, we have (see \cite[Conjecture 1.5]{Paranjape}):
\begin{conjecture}\label{Par}
Let $D$ be a smooth ample divisor on $X$. Then the restriction map for the inclusion of $D$ in $X$:
$$
\CH^p(X;\QQ) \rightarrow \CH^p(D;\QQ)
$$
is an isomorphism, for $p\leq \frac{\dim{D}-1}{2}$.
\end{conjecture}

It seems reasonable to pose the following dual of above Chow Lefschetz questions:
\begin{conjecture}\label{Nori2}
The pushforward map on the rational Chow groups, for a very general ample divisor $D\subset X$ of sufficiently large degree:
$$
j_*:\CH_k(D;\QQ) \rightarrow \CH_k(X;\QQ)
$$
is injective, whenever $k>0$.
\end{conjecture}

Similarly, we could pose the dual version of Conjecture \ref{Par}:
\begin{conjecture}\label{Par2}
Let $D$ be a smooth ample divisor on $X$. The pushforward map on the rational Chow groups, 
$$
j_*:\CH_k(D;\QQ) \rightarrow \CH_k(X;\QQ)
$$
is injective, whenever $k\,>\, \frac{\dim{D}}{2}$.
\end{conjecture}

In \S\ref{motive}, we provide a motivic interpretation of Conjectures \ref{Par} and \ref{Par2}.
If the Hodge conjecture and Bloch-Beilinson conjecture (based on the injectivity of the Abel-Jacobi map for for smooth projective 
varieties over $\ol{\QQ}$) hold, then both Conjectures \ref{Par} and \ref{Par2} hold. Concerning Conjecture \ref{Par2}, we prove the
following generalization. (See Theorem \ref{MainThmJL}.):

\begin{theorem} \label{thmL}
Assume the Hodge and Bloch-Beilinson conjectures hold. Then:
 \[
 k > \frac{\dim D -\nu}{2} \Rightarrow j_*:F^{\nu}\CH_k(D;\QQ) \hookrightarrow F^{\nu}\CH_k(X;\QQ),
 \]
where $\{F^{\nu}\CH^r(X;\QQ)\}_{\nu \geq 0}$ is the Bloch-Beilinson filtration on $\CH^r(X;\QQ)$. (The
case $\nu = 0$ yields the statement of Conjecture \ref{Par2}.)
\end{theorem}

A good source for the conjectural Bloch-Beilinson filtration is \cite{Ja}, which agrees with the filtration in \cite{JL}, under the assumptions of the
aforementioned Hodge and Bloch-Beilinson conjectures.

The motivation for the above dual Chow-Lefschetz conjectures, for us, arose while studying the following theorem by A. Collino.

Suppose $C$ is a smooth projective curve of genus $g$ over complex numbers. The symmetric power $\Sym^r(C)$ $(r>0)$ is a smooth projective variety of dimension $r$ and
fix a point $p\in C$.
The inclusion
$$
\Sym^{r-1}(C) \hookrightarrow \Sym^r(C), \,\, (x_1+x_2+...+x_{r-1}) +p \mapsto (x_1+x_2+...+x_{r-1}+p)
$$
is a smooth ample divisor \cite{ACGH}.
 
\begin{theorem}\cite[Theorem 1]{Collino}
The pushforward map on the Chow groups:
$$
\CH_k(\Sym^{r-1}(C);\QQ) \rightarrow \CH_k(\Sym^r(C);\QQ)
$$
is injective, $k\geq 0$.
\end{theorem}

This provides a prime example verifying Conjectures \ref{Nori2}, \ref{Par2} and the bounds in Theorem \ref{thmL}.

The next example is closely related to the above example on symmetric power.

In \S \ref{Incltheta}, our  aim is to verify the bounds given in Theorem \ref{thmL} when $D$ is the Theta divisor,  
on the Jacobian of a smooth projective curve. It is well-known that $\Theta$ is an ample divisor on $J(C)$.
 We state it here, as follows.

Let $C$ be a smooth projective curve of genus $g$ and let $\Th$ denote a Theta divisor inside the Jacobian $J(C)$ of $C$. Denote the inclusion
$j: \Theta \hookrightarrow J(C)$.

\begin{theorem}\label{thmI}
Assume $C$ is a non-hyperelliptic smooth projective curve of genus $g\geq 3$, over $\mathbb{C}$.
The pushforward morphisms
$$
j_*: F^1\CH_{g-2}(\Theta; \QQ) \rightarrow F^1\CH_{g-2}(J(C);\QQ)
$$
and
$$
j_*:F^2\CH_{g-3}(\Theta;\QQ) \rightarrow F^2\CH_{g-3}(J(C);\QQ)
$$
are injective.
\end{theorem}

See \S \ref{Mainthmmainsection}, Theorem \ref{Mainthmproof}.

In general, the Theta divisor is a singular variety with singular locus $B$, of dimension at least $g-4$. 
Equality holds if $C$ is non-hyperelliptic. Hence if $g\leq 3$ then $\Theta$ is smooth, and fulfils 
the above conjectural bound in Theorem \ref{thmL}. 
Furthermore, when  $g=4$ and $C$ is non-hyperelliptic (this is the generic situation), then $\Theta$ is singular and 
$B$ is a finite set  of points. The Chow groups of $\Theta$ are taken as the usual Chow groups 
$\CH_1(\Theta-B)$. The reader should be aware of the fact that the Bloch-Beilinson filtration only applies 
to smooth projective varieties. However, for our purposes, there is the Abel-Jacobi map defined on the cycles homologous 
to zero on this group, and we define $F^2\CH_{1}(\Theta;\QQ)$ as the kernel of this map. 
When $g>4$, the same convention is used
for $F^2\CH_{g-3}(\Theta;\QQ)$, i.e., 
$$
F^2\CH_{g-3}(\Theta;\QQ) := F^2\CH_{g-3}(\Theta-B;\QQ)
$$
Here the right term is the kernel of Abel-Jacobi map, see \S \ref{conv}.

\subsection{Comments on Theorem \ref{thmI}} We felt it important to incorporate some interesting comments 
from the referee regarding the above theorem. The assertion preceding Theorem \ref{thmI}, viz., to prove
Theorem \ref{thmL} in the case of Jacobians, can be construed as not as optimal as one would like. What is meant by this
is the following:

(i) There are two parameters, $k$ and $\nu$ in Theorem \ref{thmL}, and once one has proven the result for a given
pair $(k,\nu)$, one has it for all pairs $(k,\nu')$ with $\nu'\geq \nu$. So in particular, for a fixed $k$, the 
most interesting value of $\nu$ is the minimum value, $\nu = $ max$\{0,\dim D -2k+1\}$.

(ii) At the same time, since conjecturally $F^{\nu}\CH_k(D;\QQ) = F^{\nu}\CH^{\dim D -k}(D;\QQ) = 0$ for $\nu > \dim D -k$ (see 
Theorem \ref{TMot}), the case $\nu = \dim D -k$ is the smallest term of the filtration for which 
the statement is conjecturally non-trivial.

(iii) More succintly, in the case of a Jacobian and its Theta divisor, when one takes $k=g-2$, the case of most interest in
then $\nu =$ max$\{0,(g-1)-2(g-2)+1\} =$ max$\{0,4-g\}$. Note that if the curve is not hyperelliptic, then
$g\geq 3$, so that the case of most interest is $\nu=1$ for $g=3$, and $\nu=0$ for $g\geq 4$. When $k=g-3$, 
then one wants to look at $\nu = $ max$\{0,6-g\}$.  Therefore, for non-hyperelliptic curves of genus $3$, the case 
$\nu = 3$ is the case of most interest,  for genus $4$ the case $\nu = 2$ is of most
interest, for $g=5$ the case $\nu = 1$ is of most interest, and for higher genus, $\nu = 0$.

(iv) Consequently, for the first assertion of Theorem \ref{thmI}, regarding $F^1\CH_{g-2}$, for $g=3$ the 
statement is sharp with respect to Theorem \ref{thmL}, but for $g>3$ one should point out that from 
Theorem \ref{thmL}, one would really like to have the statement for $F^0\CH_{g-2}$. And for the
second assertion, regarding $F^2\CH_{g-3}$, in the case $g=3$ this choice $\nu = 2$ is stronger than what 
Theorem \ref{thmL} predicts, sharp for $g=4$, that for $g=5$, one would want $\nu=1$, and for $g\geq 6$ 
one would want $\nu=0$. At the same time, one can say that both assertions of Theorem \ref{thmI} are
made for the smallest term of the filtration on Chow for which the statement is conjecturally non-trivial.

{\small \textbf{Acknowledgements:}
The first named author is grateful to Department of Atomic Energy, India for funding this project. The third named author is partially supported by a
grant from the Natural Sciences and Engineering Research Council of Canada. The authors owe a debt of gratitude to the referee for 
going through our paper meticulously, and for the many helpful comments.}

\textbf{Notation}:
Here $k$ is an uncountable, algebraically closed field and all the varieties are defined over $k$. Denote
$$
\CH_d(X;{\mathbb Q}):= \CH_d(X)\otimes {\mathbb Q}.
$$
Here $X$ is a variety of pure dimension $n$, defined over $k$ and ${\CH}_d(X)$ denotes the Chow group of $d$-dimensional cycles modulo rational equivalence.

We write
$$
\CH_d(X,s;{\mathbb Q}) := \CH^{\dim X +s-d}(X,s) \otimes \mathbb Q,
$$
the Bloch's higher Chow groups (\cite{Bloch}) with ${\mathbb Q}$-coefficients.

\section{Motivic interpretations}\label{motive}

We wish to provide a motivic interpretation of Conjecture  \ref{Par2}.  But first some terminology,
and background material, which is specific to this section only. 
Let $\QQ(r)$ be the Tate twist and consider the category of mixed Hodge structures over $\QQ$ (MHS). For a $\QQ$-MHS V, we put
\[
\Gamma(V) = \hom_{\rm MHS}(\QQ(0),V),
\]
\[
J(V) = {\rm Ext}^1_{\rm MHS}(\QQ(0),V).
\]
For instance, if $X = X/\C$ is smooth and projective, then $\Gamma\big(H^{2r}(X,\QQ(r))\big)$ can be identified with 
$\QQ$-betti cohomology classes of Hodge type $(r,r)$, and  $J\big(H^{2r-1}(X,\QQ(r))\big)$ can be identified
(via J. Carlson) with the Griffiths jacobian (tensored with $\QQ$). There is the cycle class map
${\CH}^r(X;\QQ) \to \Gamma\big(H^{2r}(X,\QQ(r))\big)$, conjecturally surjective under the classical Hodge
conjecture (HC), with kernel $\CH^r_{\hom}(X;\QQ)$. Accordingly there
is the Griffiths Abel-Jacobi map $AJ\otimes\QQ: {\CH}^r_{\hom}(X;\QQ) \to J\big(H^{2r-1}(X,\QQ(r))\big)$.
Beilinson and Bloch have independently conjectured the following:

\begin{conjecture}[BBC]
Let $W/\ol{\QQ}$ be smooth and projective, and assume given an integer $r\geq 0$. Then the Abel-Jacobi map
\[
AJ\otimes\QQ : {\CH}_{\hom}^r(W/\ol{\QQ};\QQ) \to J\big(H^{2r-1}(W(\C),\QQ(r))\big),
\]
is injective.
\end{conjecture}
\begin{remark} If one assumes the HC + BBC, then $W/\ol{\QQ}$ can be replaced by a smooth quasi-projective
variety.
\end{remark}

Next, we need to inform the reader of the conjectured Bloch-Beilinson (BB) filtration. First conceived
by Bloch and later fortified  by Beilinson in terms of motivic extension datum, the idea is to measure the
complexity of $\CH^r(X;\QQ)$ in terms of a conjectural descending filtration. Rather than defining it here,
we provide an explicit candidate  which will define a Bloch-Beilinson filtration in the event  that the HC and BBC hold.

\subsection{A candidate BB filtration} We begin with the following result, by recalling:

\begin{theorem}[\cite{JL}]\label{TMot}
Let $X/\C$ be smooth and projective, of dimension $d$.
Then for all $r\geq 0$, there is a descending  filtration,
\[
\CH^r(X;\QQ) = F^0 \supset F^1\supset \cdots \supset F^{\nu}
\supset F^{\nu +1}\supset
 \cdots \supset F^r\supset F^{r+1}
= F^{r+2}=\cdots,
\]
which satisfies the following:

\medskip
\noindent
{\rm (i)} $F^1 = \CH^r_{\hom}(X;\QQ)$.

\medskip
\noindent
{\rm (ii)} $F^2 \subseteq \ker AJ\otimes\QQ : \CH^r_{\hom}(X;\QQ)
\rightarrow J\big(H^{2r-1}(X(\C),\QQ(r))\big)$.

\medskip
\noindent
{\rm (iii)} $F^{\nu_1}\CH^{r_1}(X;\QQ)\bullet F^{\nu_2}\CH^{r_2}(X;\QQ)\subset 
F^{\nu_1 +\nu_2}\CH^{r_1+r_2}(X;\QQ)$, where $\bullet$ is
the intersection product.

\medskip
\noindent
{\rm (iv)} $F^{\nu}$ is  preserved under the action of correspondences
between smooth projective varieties over $\C$.

\medskip
\noindent
{\rm (v)} Let $Gr_F^{\nu} :=  F^{\nu}/F^{\nu +1}$ and  
assume that the K\"unneth components of the diagonal
class $[\Delta_X] = \oplus_{p+q = 2d}[\Delta_X(p,q)] \in
H^{2d}(X\times X,\QQ(d)))$ are algebraic.  Then
\[
\Delta_X(2d-2r+\ell,2r-\ell)_{\ast}\big\vert_{Gr_F^{\nu}\CH^r(X;\QQ)}
= \delta_{\ell,\nu}\cdot \text{\rm Identity}.
\]
[If we assume the conjecture that homological and numerical equivalence coincide,
then {\rm (v)} says that $Gr_F^{\nu}$ factors through the Grothendieck motive.]

\medskip
\noindent
{\rm (vi)} Let $D^r(X) := \bigcap_{\nu}F^{\nu}$.
 If the HC, and the Bloch-Beilinson conjecture (BBC) on the injectivity of the Abel-Jacobi map ($\otimes\QQ$)
holds for smooth projective varieties defined over $\ol{\QQ}$, then $D^r(X) = 0$.
\end{theorem}

It is essential to briefly explain how this filtration comes about.
Consider a $\ol{\QQ}$-spread $\rho : \mathcal{X} \to \mathcal{S}$, where
$\rho$ is smooth and proper.  Let $\eta$
be the generic point of $\mathcal{S}$, and put $K := \ol{\QQ}(\eta)$. Write
$X_K := \mathcal{X}_\eta$. We introduced a 
decreasing filtration $\mathcal{F}^\nu\CH^r(\mathcal{X};\QQ)$, with the
property that   $Gr_\mathcal{F}^\nu\CH^r(\mathcal{X};\QQ) \hookrightarrow 
E_{\infty}^{\nu,2r-\nu}(\rho)$, (no conjectures used here!), where $E_{\infty}^{\nu,2r-\nu}(\rho)$
is the $\nu$-th graded piece of the Leray filtration on
the lowest weight part $\underline{H}^{2r}_{\mathcal H}(\mathcal{X},\QQ(r))$ 
of Beilinson's absolute Hodge 
cohomology $H^{2r}_{\mathcal H}(\mathcal{X},\QQ(r))$ associated to $\rho$.
That lowest weight part $\underline{H}^{2r}_{\mathcal H}(\mathcal{X},\QQ(r))
\subset {H}^{2r}_{\mathcal H}(\mathcal{X},\QQ(r))$ is given by the image
$H^{2r}_{\mathcal H}(\ol{\mathcal{X}},\QQ(r)) \to H^{2r}_{\mathcal H}(\mathcal{X},\QQ(r))$,
where $\ol{\mathcal{X}}$ is a smooth compactification of $\mathcal{X}$. There
is a cycle class map $\CH^{r}(\mathcal{X};\QQ)
:= \CH^{r}(\mathcal{X}/\ol{\QQ};\QQ) \to 
\underline{H}^{2r}_{\mathcal H}(\mathcal{X},\QQ(r))$, which is conjecturally
injective under the BBC  $+$ HC conjectures,
using the fact that there is a short exact sequence:
\[
0 \to J\big(H^{2r-1}(\mathcal{X},\QQ(r))\big)\to  {H}^{2r}_{\mathcal H}(\mathcal{X},\QQ(r)) \to
\Gamma\big(H^{2r}(\mathcal{X},\QQ(r))\big) \to 0.
\]
(Injectivity would imply $D^{r}(X) = 0$.)
Regardless of whether or not injectivity holds, the 
filtration $\mathcal{F}^\nu\CH^r(\mathcal{X};\QQ)$ is given by the
pullback of the Leray filtration on $\underline{H}^{2r}_{\mathcal H}(\mathcal{X},\QQ(r))$ to
$\CH^{r}(\mathcal{X};\QQ)$. 
The term $E_{\infty}^{\nu,2r-\nu}(\rho)$
fits in a short exact sequence:
$$
0\to \underline{E}_{\infty}^{\nu,2r-\nu}(\rho) \to E_{\infty}^{\nu,2r-\nu}(\rho)
\to \underline{\underline{E}}_{\infty}^{\nu,2r-\nu}(\rho) \to 0,
$$
where
\[
\underline{\underline{E}}_{\infty}^{\nu,2r-\nu}(\rho) =
\Gamma\big(H^\nu(\mathcal{S},R^{2r-\nu}\rho_{\ast}\QQ(r))\big),
\]
\[
\underline{E}_{\infty}^{\nu,2r-\nu}(\rho) = \frac{J\big(
W_{-1}H^{\nu-1}(\mathcal{S},R^{2r-\nu}\rho_{\ast}\QQ(r))\big)}{\Gamma\big(Gr_W^0H^{\nu-1}
(\mathcal{S},R^{2r-\nu}\rho_{\ast}\QQ(r))\big)}
\subset J\big(H^{\nu-1}(\mathcal{S},R^{2r-\nu}\rho_{\ast}\QQ(r))\big).
\]
[Here the latter inclusion is a result of the short exact sequence:
$$
W_{-1}H^{\nu-1}(\mathcal{S},R^{2r-\nu}\rho_{\ast}\QQ(r)) \hookrightarrow 
W_0H^{\nu-1}(\mathcal{S},R^{2r-\nu}\rho_{\ast}\QQ(r))
\twoheadrightarrow Gr_W^0H^{\nu-1}(\mathcal{S},R^{2r-\nu}\rho_{\ast}\QQ(r)).]
$$
One then has (by definition)
\[
F^{\nu}\CH^r(X_K;\QQ) =\lim_{\buildrel \to\over 
{U\subset S/\ol{\QQ}}}\mathcal{F}^\nu\CH^r(\mathcal{X}_U;\QQ), \quad \mathcal{X}_{U} := \rho^{-1}(U)
\]
\[
F^{\nu}\CH^r(X_{\C};\QQ) = \lim_{\buildrel \to
\over {K\subset \C}}F^{\nu}\CH^r(X_K;\QQ)
\]
Further, since direct limits preserve exactness,
\[
Gr_F^{\nu}\CH^r(X_K;\QQ) =
\lim_{\buildrel \to\over {U\subset S/\ol{\QQ}}}Gr_{\mathcal{F}}^\nu\CH^r(\mathcal{X}_{U};\QQ),
\]
\[
Gr_F^{\nu}\CH^r(X_{\C};\QQ) = \lim_{\buildrel \to\over 
{K\subset \C}}Gr_F^{\nu}\CH^r(X_K;\QQ)
\]

\subsection{}

Now let $j : D \hookrightarrow X$ be an inclusion of smooth irreducible projective varieties, with $D$ ample and of codimension $1$. 
The weak Lefschetz theorem implies that $j^* : H^i(X,\ZZ) \to H^i(D,\ZZ)$ is an isomorphism if $i < \dim D$ and injective for $i = \dim D$.
If we set $i = 2r-\nu$, then the statement $2r < \dim D$ implies that $2r-\nu \leq \dim D-1$ for $0\leq \nu\leq r$.  Then by Theorem \ref{TMot},
and under the assumption of the HC and BBC:
\[
r \leq \biggl[\frac{\dim D -1}{2}\biggr] \Rightarrow j^* : Gr_F^{\nu}\CH^r(X;\QQ)\xrightarrow{\sim} Gr_F^{\nu}\CH^r(D;\QQ), \ \forall \nu = 0,...,r
\]
\[
\Rightarrow j^* : \CH^r(X;\QQ) \xrightarrow{\sim}\CH^r(D;\QQ),
\]
by downward induction. This incidentally, provides the motivic 
interpretation of Conjecture \ref{Par}.\footnote{We also remark in passing that under the same conjectural assumptions
and argument, we have
\[
r \leq \biggl[\frac{\dim D -1+\nu}{2}\biggr] \Rightarrow j^* : Gr_F^{\ell}\CH^r(X;\QQ)\xrightarrow{\sim} Gr_F^{\ell}\CH^r(D;\QQ), \
\forall \ell = \nu,...,r
\]
\[
\Rightarrow j^* : F^{\nu}\CH^r(X;\QQ) \xrightarrow{\sim}F^{\nu}\CH^r(D;\QQ).
\]}

Let $(j^*)^{-1}: \CH^r(D;\QQ) \xrightarrow{\sim}\CH^r(X;\QQ)$ be the inverse map. It is
clearly cycle induced by the HC applied to the isomorphism of Hodge structures:
\[
[j^*]^{-1} : \bigoplus_{\nu=0}^r H^{2r-\nu}(D,\QQ) \xrightarrow{\sim} \bigoplus_{\nu=0}^r H^{2r-\nu}(X,\QQ).
\]
\big[\underline{Explicit}: Apply the Hodge conjecture to 
\[
\Gamma\biggl(\bigoplus_{\nu=0}^r H^{2\dim D -2r+\nu}(D,\QQ)\otimes H^{2r-\nu}(X,\QQ)\big(\dim D\big)\biggr).\big]
\]
 One clearly has a commutative diagram;
 \begin{equation}\label{JL1}
 \begin{matrix} \CH^r(D;\QQ)&\begin{matrix} \xrightarrow{\quad(j^*)^{-1}\quad \sim\quad}\\
 \xleftarrow{\quad j^*\quad\sim\quad}\end{matrix}&\CH^r(X;\QQ)\\
 &\\
 j_*\searrow&&\swarrow j_*\circ j^*\\
 &\\
 & \CH^{r+1}(X)
 \end{matrix}
 \end{equation}
 moreover $j_*\circ j^* = \cup \{D\}$. Since $j:D \hookrightarrow X$ is ample, it follows that
 for $2r < \dim X$, $j_*\circ j^* : H^{2r-\nu}(X,\QQ) \to H^{2(r+1)-\nu}(X,\QQ)$ is injective.
 Now working with the diagram:
 \begin{equation}\label{JL2}
 \begin{matrix} 0&\to&F^{\nu+1}\CH^r(X;\QQ)&\to&F^{\nu}\CH^r(X;\QQ)&\to&Gr_F^{\nu}\CH^r(X;\QQ)&\to&0\\
 &\\
&& j_*\circ j^*\biggl\downarrow\quad&& j_*\circ j^*\biggl\downarrow\quad&& j_*\circ j^*\biggl\downarrow\quad\\
&\\
0&\to&F^{\nu+1}\CH^{r+1}(X;\QQ)&\to&F^{\nu}\CH^{r+1}(X;\QQ)&\to&Gr_F^{\nu}\CH^{r+1}X;\QQ)&\to&0
 \end{matrix}
 \end{equation}
 it follows that if the left and right vertical arrows in diagram (\ref{JL2}) are injective, then so is the middle. By downward
 induction on $\nu$, we deduce from the BB filtration that $j_*\circ j^*$ in diagram (\ref{JL1}) is injective, {\it a fortiori}
 $j_*$ is injective in (\ref{JL1}). Now let $k = d-1-r = \dim D -r$. Then we have $j_* : \CH_k(D;\QQ) \to \CH_k(X;\QQ)$ injective,
 provided $k > \dim D/2$. Quite generally, one can show the following:
 
 \begin{theorem}\label{MainThmJL} Assume the Hodge (HC) and Bloch-Beilinson (BBC) conjectures. Then:
 \[
 k > \frac{\dim D -\nu}{2} \Rightarrow j_*:F^{\nu}\CH_k(D;\QQ) \hookrightarrow F^{\nu}\CH_k(X;\QQ).
 \]
 \end{theorem}
 Recall that under the assumptions, the BB-filtration is the same as Lewis' filtration.
 
 Now if we allow the injective statement $j^*: H^{d-1}(X,\QQ) \hookrightarrow H^{d-1}(D,\QQ)$, then in diagram (\ref{JL1}),
 $j^*$ is injective with left inverse $(j^*)^{-1}$. Then $2k = 2\dim D -2r \geq \dim D > \dim D -1$, i.e. $k > \frac{\dim D -1}{2}$,
but a caveat is in order here as $(j^*)^{-1}$ is not injective. We can get around this by restricting to null-homologous cycles,
via the above theorem for $\nu=1$.

The next 3 examples illustrate what can happen if
\[
\frac{\dim D -1}{2} < k \leq \frac{\dim D}{2},
\]
thus indicating that the inequality in Conjecture \ref{Par2}, is effective.

 \begin{example} Let $j:D \hookrightarrow X$ be a finite set of points defining an ample divisor on a smooth curve $X$. We assume that
 $D$ supports a zero cycle that is rationally equivalent to zero on $X$. Obviously $j_* : \CH_0(D;\QQ) \to \CH_0(X;\QQ)$ is not
 injective, and yet $k = 0 = (\dim D)/{2}$.
 \end{example}
 
 \begin{example}
 Let $j : D\hookrightarrow X := \PR^3$ be a smooth surface with Picard rank $\rho >1$, such as a Fermat surface of degree $\geq 2$. 
 Note that $\CH_1(D;\QQ) \simeq \QQ^{\rho}$ and $\CH_1(X;\QQ)
 \simeq \QQ$. Thus $j_* : \CH_1(D;\QQ) = F^0\CH_1(D;\QQ)\to F^0\CH_1(X;\QQ)=\CH_1(X;\QQ)$ is not injective. 
 Here $k = 1 = (\dim D)/2$ and $\nu = 0$. If $k=1 = \nu$, then
$ j_* : F^1\CH_1(D;\QQ) = \CH_{1,\hom}(D;\QQ)\to \CH_{1,\hom}(X;\QQ)=F^1\CH_1(X;\QQ)$ is
trivially injective since $ \CH_{1,\hom}(D;\QQ) = 0$. Here $k = 1 > (\dim D - 1)/2$.
 \end{example}
 
 \begin{example} Let $D =$ Fermat quintic in $\PR^5 =: X$. Let $\xi = L_1-L_2\in \CH_2(D;\QQ)$, a difference of
 two nonhomologous planes in $D$. Then $j_*(\xi) = 0$. Here $k=2 = (\dim D)/2$.
\end{example}

Regarding Conjecture 1.1, if $n = \dim X$ then we require $p<n-1$ for an isomorphism and $p=n-1$ for an injection.
[Consider the fact that $\CH^n(D) = 0$, and yet $\CH^n(X)$ can be highly nontrivial.]

\subsection{Higher Chow analogues}\label{highconjecture}

From the works of M. Saito and M. Asakura (see \cite{AS}), Theorem \ref{TMot} naturally extends to the higher Chow groups. In
particular, if one assumes the HC, together with a generalized version of the BBC, viz.,

\begin{conjecture}\label{BBC'} Let $W/\ol{\QQ}$ be a smooth projective variety. Then the Abel-Jacobi map
\[
\CH^r_{\hom}(W/\ol{\QQ},m;\QQ) \to J\big(H^{2r-m-1}(W,\QQ(r))\big),
\]
is injective;
\end{conjecture}
then for $X/\C$ smooth projective of dimension $d$, there is a (unique) BB filtration
\[
\{F^{\nu}\CH^r(X,m;\QQ)\}_{\nu=0}^{r},
\]
  for which the $\nu$-th graded piece
\[
Gr_F^{\nu}\CH^r(X,m;\QQ) \simeq \Delta_{X}(2d -2r+m+\nu,2r-m-\nu)_*\CH^r(X,m;\QQ).
\]

\begin{theorem} Let us assume Conjecture \ref{BBC'} and the HC. Then
\[
j^* : \CH^r(X,m;\QQ) \xrightarrow{\sim} \CH^r(D,m;\QQ),
\]
for
\[
r\leq \frac{\dim D +m-1}{2}; \quad {\rm moreover,}
\]
 \[
 k > \frac{\dim D -\nu +m}{2} \Rightarrow j_*:F^{\nu}\CH_k(D,m;\QQ) \hookrightarrow F^{\nu}\CH_k(X,m;\QQ).
 \]
\end{theorem}

\begin{proof}(Sketch.) Using the theory of mixed Hodge modules \cite{AS}, the idea of proof is virtually the same as when $m=0$, with a modification of indices.
For instance, one is now dealing with a short exact sequence
$$
0\to \underline{E}_{\infty}^{\nu,2r-\nu-m}(\rho) \to E_{\infty}^{\nu,2r-\nu-m}(\rho)
\to \underline{\underline{E}}_{\infty}^{\nu,2r-\nu -m}(\rho) \to 0,
$$
where
\[
\underline{\underline{E}}_{\infty}^{\nu,2r-\nu-m}(\rho) =
\Gamma\big(H^\nu(\mathcal{S},R^{2r-\nu-m}\rho_{\ast}\QQ(r))\big),
\]
\[
\underline{E}_{\infty}^{\nu,2r-\nu-m}(\rho) = \frac{J\big(
W_{-1}H^{\nu-1}(\mathcal{S},R^{2r-\nu-m}\rho_{\ast}\QQ(r))\big)}{\Gamma\big(Gr_W^0H^{\nu-1}
(\mathcal{S},R^{2r-\nu-m}\rho_{\ast}\QQ(r))\big)}
\subset J\big(H^{\nu-1}(\mathcal{S},R^{2r-\nu-m}\rho_{\ast}\QQ(r))\big).
\]
The statement $j^* :H^{2r-\nu-m}(X,\QQ) \xrightarrow{\sim} H^{2r-\nu-m}(D,\QQ)$ holds for all $\nu = 0,...,r$ provided
that $2r-m \leq \dim D -1$, i.e. $r\leq \frac{\dim D + m-1}{2}$. Quite generally 
\[
r\leq \frac{m+\nu+\dim D -1}{2} \Rightarrow j^* : 
F^{\nu}\CH^r(X,m;\QQ) \xrightarrow{\sim} F^{\nu}\CH^r(D,m;\QQ).
\]
For the latter part of the theorem, observe that $\CH^r(D,m) = \CH_k(D,m)$, where
$k = \dim D +m -r$. Then
\[
r\leq \frac{m+\nu+\dim D -1}{2} \Leftrightarrow k \geq \frac{\dim D +m-\nu+1}{2} \Leftrightarrow k > \frac{\dim D +m-\nu}{2}.
\]
One then argues, as in the case $m=0$, that 
\[
k > \frac{\dim D +m-\nu}{2} \Rightarrow j_*:F^{\nu}\CH_k(D,m;\QQ) \hookrightarrow F^{\nu}\CH_k(X,m;\QQ).
\]
\end{proof}

\begin{example} Let $X = \PR^2$ and $j : D\hookrightarrow X$ an elliptic curve. We consider
the map $j_* : \CH_1(D,2;\QQ) \to \CH_1(\PR^2,2;\QQ)$. In this case $k=1$ is almost, but not quite in the range
of the above theorem, even in the event that $\nu=1$, where it is well-known that for $m\geq1$
that $F^0\CH^r(X,m;\QQ) = F^1\CH^r(X,m;\QQ)$, as $\Gamma H^{2r-m}(W,\QQ(r)) = 0$, for any
projective algebraic manifold $W$.
Note that $\CH_1(D,2;\QQ)  = \CH^2(D,2;\QQ)$ and $ \CH_1(\PR^2,2;\QQ) = \CH^3(\PR^2,2;\QQ)$.
We need the following terminology. Given a variety $Y/\C$, we denote by $\pi_Y: Y \to \bSpec(\C)$ the structure
map, and where appropriate, $L_Y$ is the operation of taking the intersection product with a hyperplane section 
of $Y$. Note that by a slight generalization of the Bloch-Quillen formula, $\CH^1(Y,2) = 0$ for smooth $Y$, and for dimension
reasons, $\CH^3(\bSpec(\C),2) = 0$. Thus by the projective bundle formula, $\CH^3(\PR^2,2) = 
L_{\PR^2}\cup\pi_{\PR^2}^{*}\CH^2(\bSpec(\C),2) \simeq \CH^2(\bSpec(\C),2)$. Note that $\pi_D^* : \CH^2(\bSpec(\C),2;\QQ) \to 
\CH^2(D,2;\QQ)$ is injective. This is because, up to multiplication by some $N\in \NN$, the left inverse is given by
$\pi_{D,*}\circ L_D$. There is a commutative diagram:
\[
\begin{matrix}
0\\
\uparrow\\
{\bf cok}\\
\uparrow\\
\CH^2(D,2;\QQ)&\xrightarrow{j_*}&\CH^3(\PR^2,2;\QQ)\\
\pi_D^*\uparrow\quad\  &&|\wr\\
\CH^2(\bSpec(\C),2;\QQ)&=&\CH^2(\bSpec(\C),2;\QQ)\\
\uparrow\\
0
\end{matrix}
\]
It is obvious that ${\bf cok} \ne 0$ is the obstruction to $j_*$ being injective, and
yet that is the case if $D$ is an elliptic curve. Note that if we accommodate the situation where
$k=2$, then we are looking at $j_*:0=\CH^1(D,2) =\CH_2(D,2) \to \CH_2(\PR^2,2) =
\CH^2(\PR^2,2) \simeq \CH^2(\bSpec(\C),2) = K_2(\C)$, which is clearly injective, albeit not surjective.
\end{example}

\begin{example}
Let $X = \PR^3$, and $j:D\hookrightarrow \PR^2$ a general $K3$ surface. The map $j_*:\CH^2(D,1;\QQ) = \CH_1(D,1;\QQ)\to \CH^3(\PR^3,1;\QQ)$
is not injective, due to the presence of ``indecomposables'' in $\CH^2(D,1;\QQ)$ \cite{C-L}. Notice that $k=1 \leq 
\frac{\dim D -\nu + m}{2} = \frac{3-\nu}{2}$, for $\nu = 0, 1$. If we conside a $k=2$ example, then we are looking at
$j_* : \CH^1(D,1) = \C^{\times}  \xrightarrow{=} \C^{\times} \simeq \CH^2(\PR^3,1)$, which is an isomorphism in this case,
{\it a fortiori} $j_*$ is injective.
\end{example}



\section{Inclusion of Theta divisor into the Jacobian}\label{Incltheta}

In this section we investigate the kernel of the push-forward homomorphism, induced by the closed embedding $j$ of the Theta divisor inside the Jacobian of a smooth projective curve $C$ of genus $g$. Recall that $\Theta$ is an ample divisor on $J(C)$. 
Consider the induced pushforward map on the rational Chow groups:
$$
j_*: \CH_k(\Theta;\QQ) \rightarrow \CH_k(J(C);\QQ)
$$
for $k\geq 0$.

 To investigate the map $j_*$, we use a similar comparison theorem (\cite{Collino}) on symmetric products of the curve $C$.
 
 Fix a point $P$ in $C$. 
Consider the following map $j_C$ from $\Sym^{g-1}C$ to $\Sym^g C$
defined by
$$
P_1+\cdots+ P_{g-1}\mapsto P_1+\cdots+P_{g-1}+P\;.
$$
Here the sum denotes the unordered set of points of lengths $(g-1)$ and $(g)$.

With this definition of $j_C$ the following diagram is commutative.
$$
  \diagram
  \Sym^{g-1} C \ar[dd]_-{q_{\Theta}} \ar[rr]^-{j_C} & & \Sym^g(C) \ar[dd]^-{q} \\ \\
  \Theta \ar[rr]^-{j} & & \Pic^g(C)
  \enddiagram
  $$
  
  We recall the structure of the birational morphisms $q_\Theta$ and $q$.
  \begin{lemma}\label{BN}
  Suppose $C$ is a smooth projective curve over the complex numbers.
  
  1) The morphism $q$ is a blow-up along the subvariety
  $$
  W_g^1:=\{ l\in \Pic^g(C) : h^0(l)\geq 2\}.
  $$
  Furthermore, the singular locus of $W_g^1$ is 
  $$
  W_g^2=\{l\in \Pic^g(C) : h^0(l)\geq 3\}.
  $$
  This is a Cohen-Macaulay and a normal variety. Hence codimension of $W_g^2$ in $W^1_g$ is at least two.
  
  Denote $B=Sing(\Theta) $, the singular locus of $\Theta$.
  
  2) Then $\dim B =g-4$, when $C$ is non-hyperelliptic and  is equal to $g-3$ if $C$ is hyperelliptic.
  
  3) We have the equality:
  $$
  B= W^1_{g-1}=\{l\in \Pic^{g-1}(C) : h^0(l)\geq 2) \}
  $$
  Furthermore, $B$ is a Cohen-Maculay, and a normal variety.
  
In particular {\rm codim}$(Sing(B))\geq 2$, i.e., the singular locus of $B$ has codimension at least $2$.
 
 4) The morphism $q$ is an isomorphisms on $\Sym^g(C) -q^{-1}(W^1_g)$ onto $J(C)-W^1_g$.
 The fibres over $W^1_g$ are projective spaces.
 
 5) The morphism $q_\Theta$ is an isomorphism on $\Sym^{g-1}(C)-q_{\Theta}^{-1}(W^1_{g-1})$. 
 The fibres over $W^1_{g-1}$ are projective spaces.
  \end{lemma}
  \begin{proof}
  See \cite[p.190, Proposition 4.4, Corollary 4.5]{ACGH} and \cite[Theorem 2.3 and Lemma 1.2]{Mu}. 
  \end{proof}

  We identify $\Pic^r(C)$ with $\Pic^0(C)$ via the map $l\mapsto l \otimes \mathcal{O}_{C}(-r.p)$. 
  Apply this to $r=g-1,g$, and we obtain the commutative diagram on the rational Chow groups:
  
  \begin{equation}\label{A}
  \diagram
  \CH_k(\Sym^{g-1} C;\QQ) \ar[dd]_-{q_{\Theta *}} \ar[rr]^-{j_{C *}} & & \CH_k(\Sym^g;\QQ) \ar[dd]^-{q_*} \\ \\
  \CH_k(\Theta;\QQ) \ar[rr]^-{j_*} & & \CH_k(J(C);\QQ)
  \enddiagram
  \end{equation}
  
\subsection{ $k=0$}
 
We start by looking at the case when $k=0$. 
\begin{proposition}
\label{prop2}
Let $C$ be a smooth projective curve of genus $g$. Let $\Th$ be a Theta divisor embedded inside $J(C)$ and let $j$ denote the embedding. 
Then the push-forward homomorphism 
$$
j_*:{\CH_0(\Th;\QQ)}\rightarrow {\CH_0(J(C);\QQ)}
$$
 is injective.
\end{proposition}
\begin{proof}

Refer to the commutative diagram (\ref{A}). Consider the pushforward map:
$$
j_*: \CH_0(\Th;\QQ) \rightarrow \CH_0(J(C);\QQ).
$$
By Collino's theorem \cite[Theorem 1]{Collino}, the map $(j_C)_*$ is injective. Since  the morphisms $q$ and $q_C$ are birational morphisms (see \cite{Mu}), and $\CH_0$ is a birational invariant for smooth varieties, we have the equality:
$$
 \CH_0(J(C);\QQ)\,=\, \CH_0(\Sym^g(C);\QQ).
$$
We refer to Lemma \ref{BN}, and consider $B:=Sing(\Theta)$, the singular locus of $\Theta$ and $U:=\Theta- B$. Now $q_\Theta$ is an isomorphism outside $B$.
Consider the localization maps:
$$
\CH_0(q_{\Theta}^{-1}B)\rightarrow \CH_0(\Sym^{g-1}(C);\QQ)\rightarrow \CH_0(U;\QQ)\rightarrow 0
$$
and
$$
\CH_0(B)\rightarrow \CH_0(\Theta;\QQ) \rightarrow \CH_0(U;\QQ)\rightarrow 0.
$$

There is a stratification of $B$, on which the restriction of $q_{\Theta}$ is given by projective bundles.

By the projective bundle formula, we conlcude that 
$$
\CH_0(q_{\Theta}^{-1}(B);\QQ)= \CH_0(B;\QQ).
$$

Hence we conclude the injectivity of $j_*$.\end{proof}

\subsection{Case $g=3$, $k=1$}

Suppose $\rm{genus}(C)=3$ and $C$ is non-hyperelliptic. Here $\Sym^3(C)$ is the blow-up of $J(C)$ along the curve $C$, i.e.
$$
C= \{ L\in \Pic^3(C): h^0(L)\geq 2\}.
$$
Furthermore, $\Theta= \Sym^2(C)$.
See Lemma \ref{BN}.
Here $C\hookrightarrow J(C)\simeq \Pic^3(C)$, via $\mathcal{O}(x)\mapsto K_C\otimes \mathcal{O}_C(-x)$, where $K_C$ is the canonical line bundle of $C$.

Hence we can write 
$$
q:\Sym^3(C)= Bl_{C}(J(C)) \rightarrow J(C).
$$
Let $E_C$ denote the exceptional surface inside $\Sym^3(C)$.

By the blow-up formula:
$$
\CH_1(\Sym^3(C);\QQ) \,=\, \CH_1(J(C);\QQ) \oplus \CH_1(E_C;\QQ).
$$
 
 \begin{proposition}\label{g3} Assume $g=3$ and $C$ is non-hyperelliptic.
The pushforward map
$$
j_*: F^1\CH_1(\Theta;\QQ)\rightarrow F^1\CH_1(J(C);\QQ)
$$
is injective.
\end{proposition}
\begin{proof}
Let $\Theta $ be denoted by $H:=[\Theta] \in \CH^1(J(C))$. Consider the maps given by intersection with the $\Theta$, in $J(C)$:
$$
\CH^1(J(C);\QQ) \stackrel{\cap H}\rightarrow \CH^1(\Theta;\QQ) \rightarrow \CH^2(J(C);\QQ).
$$
This map restricts on the $F^1$-piece, and is compatible with the Abel-Jacobi maps. Hence we get a commutative diagram:
$$
\diagram
F^1\CH^1(J(C);\QQ) \ar[dd]_-{AJ^1_J} \ar[r]^-{\cap H}& F^1\CH^1(\Theta;\QQ) \ar[dd]_-{AJ^1_\Theta} \ar[r]^-{j_*}& F^1\CH^2(J(C);\QQ)\ar[dd]_-{AJ^2_J} \\ \\
IJ(H^1(J(C);\QQ) \ar[r]^-{\cap H}  & IJ(H^1(\Theta;\QQ)) \ar[r]^{h}& IJ(H^3(J(C);\QQ) 
\enddiagram
$$
Now $AJ^1_J$ and $AJ^1_\Theta$ are isomorphisms. Since $H$ is an ample divisor,  $h\circ (\cap H)$ is injective by the hard Lefschetz theorem, and $\cap H$ on $IJ(H^1(J(C);\QQ)$ is an isomorphism, by Lefschetz hyperplane theorem. This implies that $h$ is injective. Hence $j_*$ is injective.
\end{proof} 

\section{Abel-Jacobi maps on $F^1\CH_r(\Theta;\QQ)$}

When $g\geq 4$, the theta divisor is singular and the singular locus $B$ has dimension at least $g-4$. When $C$ is non-hyperelliptic the dimension is equal to $g-4$.  See Lemma \ref{BN}. We will consider a non-hyperelliptic curve $C$, which is the generic situation.

Consider the localization sequence:
$$
\CH_k(B;\QQ) \rightarrow \CH_k(\Theta;\QQ) \rightarrow \CH_k(\Theta-B) \rightarrow 0.
$$
We would like to know how $F^1$ behaves with localization and associate Abel-Jacobi maps to the $F^1$-terms.


\subsection{General Abel-Jacobi maps}\label{generalAJ}

Suppose $X$ is a smooth quasi-projective variety defined over the complex numbers. 
Let $X\subset \overline{X}$ be a smooth compactification of $X$.
Let MHS denote the category of $\QQ$-mixed Hodge structures. 
There is an Abel-Jacobi map:
\begin{eqnarray*}
\CH_{\hom}^m(X;\QQ) & \rightarrow & \rm{Ext}^1_{\rm{MHS}}(\QQ(0),H^{2m-1}(X,\QQ(m))) \\
              & = & \rm{Ext}^1_{\rm{MHS}}(\QQ(0),W_0H^{2m-1}(X,\QQ(m)))
\end{eqnarray*}

We are interested in the Abel-Jacobi map restricted to the image:
\[
\CH_{\hom}^m(X;\QQ)^{\circ} := {\rm Im}\big(\CH_{\hom}^m(\overline{X};\QQ)\rightarrow \CH_{\hom}^m(X;\QQ)\big).
\]
The conjectured equality
\[
\CH^m_{\hom}(X;\QQ)^{\circ} = \CH^m_{\hom}(X;\QQ),
\]
is a consequence of the Hodge conjecture. In fact, we have:

\begin{proposition} \label{HC} 
Let $Y\subset \ol{X}$ be a subvariety, where $\ol{X}$ is
smooth projective of dimension $n$, and let $X = \ol{X} \backslash Y$. For $m \leq 2$ and $m\geq n-1$ (and more generally for
all $m$ if one assumes the Hodge conjecture), there is an exact sequence:
\[
\CH^m_{{Y}}(\ol{X};\QQ)^{\circ} \xrightarrow{} \CH_{\hom}^m(\ol{X};\QQ)\to  \CH^m_{\hom}(X;\QQ)\to 0,
\]
where 
\[
\CH^m_{{Y}}(\ol{X};\QQ) = \CH_{n-m}(Y;\QQ), 
\]
and
\[
\CH^m_{{Y}}(\ol{X};\QQ)^{\circ} := \{\xi\in \CH^m_{{Y}}(\ol{X};\QQ)\ |\
j(\xi)\in  \CH_{\hom}^m(\ol{X};\QQ)\}.
\]
Here $j$ is the map
\[
\CH^m_{{Y}}(\ol{X};\QQ) \xrightarrow{j} \CH^m(\ol{X};\QQ)\to  \CH^m(X;\QQ)\to 0.
\]

\end{proposition}

\begin{proof}  Let $\xi\in  \CH^m_{\hom}(X;\QQ)$, and choose $\ol{\xi}\in  \CH^m(\ol{X};\QQ)$
which maps to $\xi$. By construction, the fundamental class $[\ol{\xi}]\in H^{2m}(\ol{X},\QQ(m))$
lies in the image $H_Y^{2m}(\ol{X},\QQ(m)) \to H^{2m}(\ol{X},\QQ(m))$. Let $q=$ codim$_{\ol{X}}Y$,
and let $\sigma : \tilde{Y}\to Y$ be a desingularization. By a weight argument and mixed Hodge theory, 
$[\ol{\xi}]$ lies in the image $H^{m-q,m-q}(\tilde{Y},\QQ(m-q)) \to H^{2m}(\ol{X},\QQ(m))$, which will come from
the fundamental class of an algebraic cycle $\gamma\in \CH^{m-q}(\tilde{Y};\QQ)$, provided that the Hodge conjecture
holds for $\tilde{Y}$. Assuming this, then $\ol{\xi} - j(\sigma_*(\gamma))\in \CH_{\hom}^m(\ol{X};\QQ)$ 
maps to $\xi\in \CH_{\hom}^m(X;\QQ)$. The rest is clear.
\end{proof}

Thus we get a map $\CH^m(X;\QQ)_{\hom}^{\circ}\rightarrow$
\begin{equation}\label{EEE}
 {\rm Im}\big(\rm{Ext}^1_{\rm{MHS}}(\QQ(0),W_{-1}H^{2m-1}(X,\QQ(m)) 
\to \rm{Ext}^1_{\rm{MHS}}(\QQ(0),W_0H^{2m-1}(X,\QQ(m))\big),
\end{equation}
where we use the fact that 
\[
W_{-1}H^{2m-1}(X,\QQ(m))  =\ {\rm Image}\big(H^{2m-1}(\ol{X},\QQ(m)) \to H^{2m-1}(X,\QQ(m))\big),
\]
together with the Abel-Jacobi image of a class in $\CH^m_{\hom}(X;\QQ)$ being in the Abel-Jacobi image
of a class in $\CH_{\hom}^m(\ol{X};\QQ)$.
Note that the term in (\ref{EEE}) above can be identified with
\[
 \frac{\rm{Ext}^1_{\rm{MHS}}(\QQ(0),W_{-1}H^{2m-1}(X,\QQ(m))}{\delta \hom_{\rm{MHS}}(\QQ(0),Gr_0^W H^{2m-1}(X,\QQ(m)))},
\]
where $\delta$ is the connecting homomorphism in the long exact sequence associated to
$$
0\rightarrow W_{-1}H^{2m-1}({X},\QQ(m))\rightarrow W_0H^{2m-1}(X,\QQ(m))\rightarrow Gr_0^WH^{2m-1}(X,\QQ(m))\rightarrow 0.
$$
In case, $Gr_0^WH^{2m-1}(X,\QQ(m))=0$ then the target of the Abel-Jacobi map is the group 
$\rm{Ext}^1_{\rm{MHS}}(\QQ(0),W_{-1}H^{2m-1}(X,\QQ(m))$.


\subsection{$ F^1$-term of $\CH_k(\Sym^{g-1};\QQ)$ and  $\CH_k(\Theta;\QQ)$ }

In this subsection, we consider the situation $B\subset \Theta$. Denote the complement $U:= \Theta -B$.

Recall from Lemma \ref{BN}, 
that the morphism $q_\Theta:\Sym^{g-1}C \rightarrow \Theta $ is a birational morphism and is a  smooth resolution of $\Theta$ (\cite{Mu}).

Denote $U_{\rm sym}:= {q_\Theta}^{-1}(U)$ and $Y:= \Sym^{g-1}(C)-U_{\rm sym}$. Note $U_{\rm sym}\simeq U$.

\begin{lemma}
The restriction map
$$
F^1\CH_k(\Sym^{g-1}(C);\QQ) \rightarrow F^1\CH_k(U_{\rm sym};\QQ)
$$
is surjective, when $k=g-2,\,g-3$. 
\end{lemma}
\begin{proof}
Apply the localization sequence and Proposition \ref{HC}, to the triple 
$$
(Y\subset \Sym^{g-1}(C)\supset U_{\rm sym}).
$$
This gives a surjective map, when $k=g-2,\,g-3$:
\begin{equation}\label{kerpush}
F^1\CH_k(\Sym^{g-1}(C);\QQ)\rightarrow F^1\CH_k(U_{\rm sym} ;\QQ)\rightarrow 0.
\end{equation}
\end{proof}

Consider the pushforward $q_{\Theta *}: \CH_k(\Sym^{g-1}(C) ;\QQ ) \rightarrow \CH_k(\Theta;\QQ)$. 
  
\begin{lemma}
For $k=g-2,g-3$, the restriction map in \eqref{kerpush}  induces a map
$$
\frac{F^1\CH_k(\Sym^{g-1}(C);\QQ)}{F^1\cap \ker(q_{\Theta *})} \rightarrow F^1\CH_k(U_{\rm sym} ;\QQ)
$$
which is an isomorphism.
\end{lemma}
\begin{proof}
Since $C$ is a non-hyperelliptic curve, $dim (B)=g-4$. Hence if $k> g-4$, the restriction 
$$
h_*:\CH_k(\Theta;\QQ)\rightarrow \CH_k(U;\QQ) 
$$
is an isomorphism.
Consider the commutative diagram:

$$
\diagram
\CH_k(\Sym^{g-1} C;\QQ) \ar[rr]^-{j_{u*}} \ar[dd]^-{q_{\Th*}} & & \CH_k(U_{\rm sym};\QQ) \ar[dd]^-{q_{u*}} \\ \\
 \CH_k(\Theta;\QQ) \ar[rr]^-{h_*} & & \CH_k(U;\QQ)
\enddiagram
$$
This induces a corresponding diagram on the $F^1$-terms. Note that $q_{u*}$ is an isomorphism, since $q_{\Theta}$ is an isomorphism outside $B$.
Hence, we obtain  an isomorphism:
\begin{equation}\label{local}
\frac{F^1\CH_k(\Sym^{g-1}(C);\QQ)}{F^1\cap \ker(q_{\Theta *})} \rightarrow F^1\CH_k(U_{\rm sym} ;\QQ)= F^1\CH_k(\Theta;\QQ)
\end{equation}

\end{proof}


\subsection{Abel-Jacobi maps on $F^1$-terms}\label{conv}

When $k=g-2,\,g-3$, denote $l=1,2$ the corresponding codimension. Using \S \ref{generalAJ} and purity of Hodge structures (here $\dim B =g-4$), there is an Abel-Jacobi map:

\begin{equation}\label{AJU}
AJ_{\Theta}:\,F^1\CH_k(\Theta;\QQ)=F^1\CH_k(U;\QQ) \rightarrow IJ(H^l(U);\QQ),
\end{equation}
where we recall that:
$$
F^2\CH_k(X;\QQ)=F^2\CH_k(U;\QQ)\,=\, \rm{kernel} (AJ_\Theta).
$$

 We recall the following, which will be used in the next section.
 
 \begin{lemma}\label{JLlemma}
 The Abel-Jacobi map
 $$
 AJ: F^1CH^1(U,\QQ) \rightarrow  IJ(H^1(U,\QQ))
 $$
 is an isomorphism.
 \end{lemma}
 \begin{proof}
 See \cite[Proposition 2.5]{JL}.
 \end{proof}


\section{Main Theorem}\label{Mainthmmainsection}

Now we can state our main theorem.

\begin{theorem}\label{Mainthmproof}
Assume $C$ is a non-hyperelliptic smooth projective curve of genus $g\geq 3$, over $\mathbb{C}$.
The pushforward morphisms
$$
j_*: F^1\CH_{g-2}(\Theta;\QQ) \rightarrow F^1\CH_{g-2}(J(C);\QQ)
$$
and
$$
j_*:F^2\CH_{g-3}(\Theta;\QQ) \rightarrow F^2\CH_{g-3}(J(C);\QQ)
$$
are injective.
\end{theorem}

Note that the first injectivity statement generalizes Proposition \ref{g3}.

\begin{proof}
Consider the birational morphisms
$$
q_\Theta: \Sym^{g-1}(C) \rightarrow \Theta
$$
and 
$$
q:\Sym^g(C)\rightarrow J(C).
$$
These maps induce the commutative diagram on the $F^1$-terms of the rational Chow groups:

  $$
  \diagram
  F^1\CH_k(\Sym^{g-1} (C);\QQ) \ar[dd]_-{q_{\Theta*}} \ar[rr]^-{j_{C*}} & & F^1 \CH_k(\Sym^g(C);\QQ) \ar[dd]^-{q_*} \\ \\
  F^1\CH_k(\Theta;\QQ) \ar[rr]^-{j_*} & &F^1 \CH_k(J(C);\QQ)
  \enddiagram
  $$

Denote $l=g-1-k$.
The above diagram is compatible via Abel-Jacobi maps to the corresponding commutative digram of intermediate Jacobians:

$$
  \diagram
  IJ(H^l(\Sym^{g-1} (C);\QQ)) \ar[dd]_-{q_{\Theta*}} \ar[rr]^-{j_{C*}} & & IJ(H^{l+2}((\Sym^g(C);\QQ)) \ar[dd]^-{q_*} \\ \\
  IJ(H^l(\Theta;\QQ)) \ar[rr]^-{h} & & IJ(H^{l+2}((J(C);\QQ))
  \enddiagram
  $$
  In the above diagram, $l=1,2$.
Since $\codim(B)\geq 2$ in $\Theta$, without any confusion, we write  $IJ(H^l(\Theta;\QQ))= IJ(H^l(U;\QQ))$.

\textbf{Case 1) $l=1$}.  

Denote $H$  the ample divisor $\Theta$ on $J(C)$, 
  Consider the diagram obtained by intersecting with $H$ in $\CH^*(J(C))$ (resp. in $H^*(J(C),\QQ))$.
$$
\diagram
F^1\CH^1(J(C);\QQ) \ar[dd]_-{AJ^1_J} \ar[r]^-{\cap H}& F^1\CH^1(\Theta;\QQ) \ar[dd]_-{AJ_\Theta} \ar[r]^-{j_*}& F^1\CH^2(J(C);\QQ)\ar[dd]_-{AJ^2_J} \\ \\
IJ(H^1(J(C);\QQ) \ar[r]^-{\cap H}  & IJ(H^1(\Theta;\QQ)) \ar[r]^{h}& IJ(H^3(J(C);\QQ) 
\enddiagram
$$
Now $AJ^1_J$ and $AJ_\Theta$ are isomorphisms. In particular $AJ_\Theta$ is defined in terms of $U := \Theta - B$ (see Lemma \ref{JLlemma}). 
Since $H$ is an ample class,  $h\circ (\cap H)$ is injective by the hard Lefschetz theorem, and $\cap H$ on $IJ(H^1(J(C);\QQ)$ is an
 isomorphism, by the Lefschetz hyperplane theorem. This implies that $h$ is injective. Hence $j_*$ is injective.

\textbf{Case 2):  $l=2$}.

  Now $q:\Sym^g(C)\rightarrow J(C)$ is a blow-up morphism along a codimension two subvariety 
  $$
  W= \,W^1_g=\, \{L\in {\Pic}^g(C): h^0(L)\geq 2\}.
  $$
 (See Lemma \ref{BN}).  Denote $E_W\subset \Sym^g(C)$ the exceptional locus of the blow-up morphism $q$.
   In particular, we can write a decomposition:
   \begin{equation}\label{CHB}
   \CH_k(\Sym^g(C);\QQ)= \CH_k(J(C);\QQ)\oplus \CH_k(E_W;\QQ)
  \end{equation}
  
  Denote $H$ the ample divisor $\Sym^{g-1}(C)$ on $\Sym^{g}(C)$, in $\CH^*(\Sym^g(C);\QQ)$ (resp. in $H^*(\Sym^g(C),\QQ))$.
  
Consider the Abel-Jacobi maps on the $F^1$-terms of the Chow groups of the symmetric products, which are compatible with the intersection $\cap H$:
$$
\diagram
F^1\CH^2(\Sym^{g}(C);\QQ) \ar[dd]_-{AJ^1_{sym}} \ar[r]^-{\cap H}& 
F^1\CH^2(\Sym^{g-1}(C);\QQ) \ar[dd]_-{AJ_{sym} }\ar[r]^-{j_*}& F^1CH^{3}(\Sym^g(C);\QQ)\ar[dd]_-{AJ^2_{sym}} \\ \\
IJ(H^3(\Sym^g(C);\QQ) \ar[r]^-{\cap H}  & IJ(H^2(\Sym^{g-1}(C);\QQ)) \ar[r]^{h}& IJ(H^{5}(\Sym^g(C);\QQ) .
\enddiagram
$$

By \cite[Theorem 1]{Collino}, the Chow restriction map $\cap H$ is surjective and $j_*$ is injective.
This implies that  using the decomposition in \eqref{CHB}, we can write the above commutative diagram as
{\Small
$$
\diagram
F^1\CH^2(J(C);\QQ) \oplus F^1\CH^1(E_W;\QQ)\ar[dd]_-{AJ^1_{sym}} \ar[r]^-{\cap H}& H.F^1\CH^2(\Sym^g(C);\QQ) \oplus H.F^1\CH^1(E_W;\QQ)\ar[dd]_-{AJ_{sym} }\ar[r]^-{j_*}& F^1\CH^{3}(\Sym^g(C);\QQ)\ar[dd]_-{AJ^2_{sym}} \\ \\
IJ(H^3(\Sym^g(C);\QQ) \ar[r]^-{\cap H}  & IJ(H^3(\Sym^{g-1}(C);\QQ)) \ar[r]^{h}& IJ(H^{5}(\Sym^g(C);\QQ) 
\enddiagram
$$
}
A similar decomposition exists for the intermediate Jacobians. This implies that we have the equality
$$
\rm{Kernel}(AJ_{\rm sym})= \rm{Kernel} (AJ_{\rm sym}^J) \oplus \rm{Kernel} (AJ_{sym}^W).
$$
Here $AJ_{\rm sym}^J$ is the restriction of $AJ_{\rm sym} $ on the first summand and $AJ_{\rm sym}^W$ is the restriction of $AJ_{\rm sym}$ on the second summand.
 
 However, 
 $$
  AJ_{\rm sym}^W: H.F^1\CH^1(E_W;\QQ) \rightarrow H.IJ(H^1(E_W))
  $$
  has no kernel.
  
Hence,
$$
\rm{Kernel}(AJ_{\rm sym})= \rm{Kernel} (AJ_{\rm sym}^J)
$$

In other words, if we consider the composed map
$$
F^1\CH^2(\Sym^{g-1}(C);\QQ) \rightarrow F^1\CH^2(\Sym^g(C);\QQ) \rightarrow F^1\CH^3(J(C);\QQ)
$$
(the second map is the projection to its first summand), then it induces an injective map
$$
F^2\CH^2(\Sym^{g-1}(C);\QQ)  \hookrightarrow F^2\CH^3(J(C);\QQ)
$$
 Now observe that 
 $$
 F^2\CH^2(\Sym^{g-1}(C);\QQ) = F^2\biggl(  \frac{\CH^2(\Sym^{g-1}(C);\QQ)}{\ker(q_{\Theta*})}\biggr).
 $$
 This is because $\ker(q_{\Theta*})$ is supported on $\Theta-B$, and
 $$
 q_{\Theta*}: \CH^2(\Sym^{g-1}(C);\QQ) \rightarrow \CH^2(\Theta;\QQ)
 $$
 is injective on  the first summand  $H.F^1\CH^2(\Sym^g(C);\QQ)$. 

 It now suffices to show that $F^2(H.\CH^1(E_W;\QQ))=0$, to conclude
$$
F^2\CH^2(\Theta;\QQ) \hookrightarrow F^2\CH^2(J(C);\QQ)
$$
is injective, for $g\geq 4$ and $C$ non-hyperelliptic. 

 \end{proof}

\begin{lemma}
$$
F^2\CH_{g-2}(E_W;\QQ) =0
$$
\end{lemma}
\begin{proof} Now $\dim(E_W)=g-1$, which is a bundle of projective spaces over $W$.

Using Lemma \ref{BN} and $\codim(W)=2$ in $J(C)$, (for a hyperplane class $h$ on $E_W$) we can write:
$$
\CH_{g-2}(E_W;\QQ) = \CH_{g-2}(W;\QQ).h \oplus \CH_{g-3}(W;\QQ) \,\,(\rm{modulo \,the\, image}\, \CH_{g-2}(E_{W^2_g};\QQ)). 
$$
Since $dim(E_{W^2_g})\leq g-2$, $F^1\CH_{g-2}(E_{W^2_g};\QQ)=0$.
 
Restricting to $F^1$-terms gives:
$$
F^1\CH_{g-2}(E_W;\QQ) =  F^1\CH_{g-3}(W)= F^1\CH^1(W;\QQ).
$$
Furthermore,  $\codim(Sing(W)) \geq 2$ (see Lemma \ref{BN},1)).

Now we are reduced to the case 1) situation when $l=1$. Namely, there is an 
Abel-Jacobi map $W- Sing(W)$, which is an isomorphism onto $IJ(H^1(W-Sing(W)))$ (see Lemma \ref{JLlemma}).  
This shows the kernel of the Abel-Jacobi map is trivial, and $F^2 \subseteq \ker AJ$, one has that
$F^2=0$. This suffices to conclude the proof. 
 \end{proof}


\end{document}